\documentclass[a4paper,12pt]{article}

\usepackage{color}
\usepackage[pdftex]{graphicx}
\usepackage{amsmath,amssymb,amsthm}
\usepackage{wrapfig}
\usepackage{ifpdf}

\newtheorem{theorem}{Theorem}

\newtheorem{lemma}{Lemma}

\newcommand{\sgn}{\mathop{\rm sign}\nolimits}

\newcommand{\nl}{\mathop{\rm NL}\nolimits}

\DeclareMathOperator{\dg2}{deg_2}

\title {Sperner type lemma for quadrangulations}

\author {Oleg R. Musin\thanks{This research is partially supported by NSF grant DMS - 1101688.}}

\begin{document}

	\ifpdf \DeclareGraphicsExtensions{.pdf, .jpg, .tif, .mps} \else
	\DeclareGraphicsExtensions{.eps, .jpg, .mps} \fi	
	
\date{}
\maketitle

\begin{abstract}  Sperner's lemma states that every Sperner coloring of a triangulation of a simplex contains a fully colored simplex. We present  a generalization of this lemma where instead of triangulations are considered quadrangulations.   
\end{abstract}

\medskip

\noindent {\bf Keywords:} Sperner's lemma, quadrangulation,  degree of mapping

\section{Introduction}



{\it Sperner's lemma} is   a statement about labellings (colorings) of triangulated simplices ($d$-balls). It is a discrete analog of the Brouwer fixed point theorem.

Let $S$ be a $d$-dimensional simplex  with vertices $v_1,\ldots,v_{d+1}$. Let $T$ be a triangulation of $S$. Suppose that each vertex of $T$ is assigned a unique label from the set $\{1,2,\ldots,d+1\}$. A labelling $L$ is called {\it Sperner's} if  the vertices are labelled in such a way that a vertex   of $T$ belonging to the interior of a face $F$ of $S$ can only be labelled by $k$ if $v_k$ is on $S$.   

 \medskip

\noindent{\bf (Sperner's lemma \cite{Sperner})} 
	{\it Every Sperner labelling of a triangulation of a $d$-dimensional simplex contains a cell labelled with a complete set of labels: $\{1,2,\ldots, d+1\}$.}

\medskip

The two-dimensional case is the one referred to most frequently. It is stated as follows:\\
{\it Given a triangle $ABC$, and a triangulation $T$ of the triangle. The set $V(T)$ of vertices of $T$ is colored with three colors in such a way that 
(i) $A, B$ and $C$ are colored 1, 2 and 3 respectively; 
(ii)   Each vertex on an edge of $ABC$ is to be colored only with one of the two colors of the ends of its edge. Then there exists a triangle from $T$, whose vertices are colored with the three different colors.}

\medskip

There are several extensions of Sperner' lemma for quadrangulations (cubical decompositions), see   \cite{KyFan60,Kuhn,Shashkin,Wolsey}. In particular, Ky Fan \cite{KyFan60} proved that any Sperner 0--1 labelling of a $d$-pile with the {neighbourhood property} contains a fully labelled $d$-cube.   
Shashkin \cite{Shashkin} extended Ky Fan's formula for local degrees of simplicial maps in \cite{KyFan67} to cubical maps. (Actually, he proved it only for small dimensions $d\le4$.) 

\medskip

In two dimensions a {\it quadrangulation}  is the division of a surface or plane polygon into a set of quadrangles with the restriction that each quadrangle side  either is entirely shared by two adjacent quadrangles or lies on the boundary. 

Our main example is the {\it pile of cubes} $\Pi_d(n_1,\ldots,n_d)$, see details in \cite[Chapter 5]{Ziegler}. It is the polytopal complex formed by all unit cubes with integer vertices in the $d$-box 
$$
B(n_1,\ldots,n_d):=\{x=(x_1,\ldots,x_d) \in {\Bbb R}^d:0\le x_i\le n_i \mbox{ for } 1\le i \le d\}. 
$$
So each $k$-cell, $0\le k\le d$, of this complex is a $k$-dimensional unit cube.

In Section 2 we consider Theorem 1 (oriented case) and Theorem 2 (non-oriented case). Both theorems  yield the following extension of Sperner's lemma for  piles in two dimensions.

\medskip

\noindent{\bf Theorem A.}
{\it Given a pile $P:=\Pi_2(m,n)$ with corners $ABCD$.  The set  of vertices of $P$ is colored with four colors in such a way that $A, B,C$ and $D$ are colored 1, 2, 3 and 4 respectively; and 
each vertex on an edge of $ABCD$ is to be colored only with one of the two colors of the ends of its edge. Then either there exists a quadrangle from $P$, whose vertices are colored with the four different colors or there exists an edge of $P$ whose two ends are colored with $(1,3)$ or $(2,4)$.}

\medskip

\noindent{\bf Remark.} In fact, for a 2-pile the {\it neighbourhood property} in \cite{KyFan60} means that there are no edges whose two vertices are colored with $[1,3]$ or $[2,4]$. So Theorem A first proved by Ky Fan.

\medskip
 
\medskip

\begin{center}
\begin{picture}(320,160)(-100,-80)
\put(-130,-70){Figure 1: Sperner's labelling of $\Pi_2(4,3)$.  One edge is colored with $(1,3)$.}

\multiput(-70,-40)(60,0){5}%
{\line(0,1){120}}

\multiput(-70,-40)(0,40){4}%
{\line(1,0){240}}

\put(-80,-50){$A$}
\put(173,-50){$B$}
\put(173,80){$C$}
\put(-83,80){$D$}

\put(-66,-36){1}
\put(160,-36){2}
\put(160,68){3}
\put(-66,68){4}

\put(-66,4){1}
\put(-66,44){4}

\put(160,4){2}
\put(160,44){3}

\put(-6,-36){2}
\put(-6,4){1}
\put(-6,44){4}
\put(-6,68){3}

\put(54,-36){2}
\put(54,4){1}
\put(54,44){3}
\put(54,68){3}

\put(114,-36){1}
\put(114,4){2}
\put(114,44){3}
\put(114,68){4}

{\linethickness{0.7mm}
\put(50,0){\line(0,1){40}}
}
 
\end{picture}
\end{center}

\medskip

We will call a quadrangles (2-cell)  whose vertices are colored with the four different colors [1,2,3,4]  and   edges (1-cells) whose two vertices are colored with $[1,3]$ or $[2,4]$ as {\it centrally labelled} cells. 

The main reason for this name is the following. 
Let $C$ be a quadrangle whose vertices $V(C):=v_1,v_2,v_3$ and $v_4$ that are labelled with 1, 2, 3 and 4 respectively. Then among all cells with vertices in $V(C)$ there are only three cells: $v_1v_2v_3v_4$ (2-cell), and two diagonals (1-cells) $v_1v_3$ and $v_2v_4$ that contain the center of $C$ inside.

We can say it in another way. We assume that the colors 1, 2, 3, 4 correspond to the labels $(-1,-1), (1,-1),(1,1),(-1,1)$ respectively. Then centrally labelled cells are cells with zero sum of its labels.

In two dimensions, the labelling in Theorem A is called Sperner's labelling.  
In other words, Theorem A states:

\medskip

\noindent{\it Any Sperner labelling of the pile $\Pi_2(m,n)$ contains at least one centrally labelled cell.}

\medskip

Theorem A can be extended for all dimensions. Let $C^d=\Pi_d(1,\ldots,1)$ denote the unit cube in ${\Bbb R}^d$.  The  vertex set of $C^d$ consists of all the $2^d$ binary $d$-tuples, and where two
vertices are adjacent if the corresponding $d$-tuples differ in exactly one coordinate position. Therefore, there is one-to-one correspondence between the vertex set $V(C^d)$ and the set  of binary numbers from 0 to $2^d-1$, i. e. the labelling $c:V(C^d)\to \{0,1,\ldots,2^d-1\}$ is bijection.

Note that corners of the pile $P:=\Pi_d(n_1,\ldots,n_d)$ can be considered as vertices of  $C^d$. 
We say that a labelling  $L:V(P)\to V(C^d)$ is {\it Sperner's} if the set of corners of  $P$ are labelled with the correspondent labels from $V(C^d))$ and vertices on the boundary of $P$ are labelled in such a way that a vertex belonging to the interior of a cell (face) $F$ of $C^d$ can be only labelled with $v$ where $v$ is a vertex of $F$.  

In Section 2 are defined centrally labelled cells for all dimensions, see Definition 2. 

\medskip

{\bf Theorem B.} {\it Any Sperner labelling of the pile $\Pi_d(n_1,\ldots,n_d)$ contains at least one centrally labelled cell.}

\medskip

Now extend Theorem A for quadrangulations.  We found a weaker assumption than Sperner's coloring. 

Let $L:V\to\{1,2,3,4\}$ be a labelling of a set $V:=\{v_1,\ldots, v_m\}$ in  a circle such that any two adjacent vertices cannot have labels (1,3) or (2,4).  Let  $$\deg([1,2],L):=p_*-n_*,$$ where $p_*$ (respectively, $n_*$) is the number of (ordering) pairs $(v_k,v_{k+1})$ such that $L(v_k)=1$ and $L(v_{k+1})=2$ (respectively, $L(v_k)=2$ and $L(v_{k+1})=1$). 

\medskip

For instance, let $L=(12212341234211223412341)$. Then $p_*=5$ and $n_*=2$. Thus, $\deg([1,2],L)=5-2=3.$

\medskip

Note that, instead of $[1,2]$ we can take $[2,3],\, [3,4]$ or $[4,1]$. Namely, we have  $$\deg([1,2],L)=\deg([2,3],L)=\deg([3,4],L)=\deg([4,1],L).$$
This  fact proved in \cite[Lemma 2.1]{MusArSp}.

Let $Q$ be a quadrangulation of a  polygon $M$. Denote by $\partial Q$ the boundary of $Q$. Then $\partial Q$ is a polygonal contour with vertices $v_{1},\ldots,v_m$ that can be considered as points in a circle.  (We assume that these vertices are in counterclockwise order.) 

Let $L:V(Q)\to\{1,2,3,4\}$ be a labelling. This labelling implies the labelling  $L_0:\partial Q\to\{1,2,3,4\}$. We  consider the set of  {\it neighboring labellings} $\nl(\partial Q)$. We write $L\in\nl(\partial Q)$ if  any edge on the boundary $\partial Q$ have no labels (1,3) or (2,4). Then $\deg([1,2],L_0)$ is well defined.  Denote $$\deg(L,\partial Q):=\deg([1,2],L_0).$$

\medskip

{\bf Theorem C.} {\it Let $Q$ be a quadrangulation of a  polygon $M$. Suppose $L:V(Q)\to\{1,2,3,4\}$ be a labelling such that $L\in\nl(\partial Q)$. Then $Q$ contains at least $|\deg(L,\partial Q)|$ centrally labelled cells.}

\medskip

\medskip
 
\medskip

\begin{center}
\begin{picture}(320,160)(-100,-80)
\put(-130,-70){Figure 2: Since $\deg(L,\partial Q)=2$,  there are  two centrally labelled cells.}

\multiput(-70,-40)(60,0){5}%
{\line(0,1){120}}

\multiput(-70,-40)(0,40){4}%
{\line(1,0){240}}

\put(-70,40){\line(-1,0){60}}
\put(-70,0){\line(-1,0){60}}
\put(-130,40){\line(0,-1){40}}

\put(-66,-36){1}
\put(160,-36){1}
\put(160,68){1}
\put(-66,68){1}

\put(170,80){\line(1,0){60}}
\put(170,40){\line(1,0){60}}
\put(170,0){\line(1,0){60}}
\put(230,0){\line(0,1){80}}

\put(220,4){1}
\put(220,44){1}
\put(220,68){2}

\put(-126,4){3}
\put(-126,28){3}

\put(-66,4){2}
\put(-66,44){2}

\put(160,4){4}
\put(160,44){1}

\put(-20,-36){2}
\put(-20,4){1}
\put(-6,44){1}
\put(-6,68){1}

\put(54,-36){3}
\put(54,4){4}
\put(54,44){1}
\put(54,68){3}

\put(114,-36){4}
\put(114,4){1}
\put(114,44){1}
\put(114,68){2}

{\linethickness{0.7mm}
\put(50,40){\line(0,1){40}}
}

\multiput(0,-30)(10,0){5}%
{\circle*{2}} 
\multiput(0,-20)(10,0){5}%
{\circle*{2}}
\multiput(0,-10)(10,0){5}%
{\circle*{2}}

{\linethickness{0.5mm}
\put(-10,-40){\line(1,0){60}}
\put(-10,0){\line(1,0){60}}
\put(-10,-40){\line(0,1){40}}
\put(50,-40){\line(0,1){40}}
}
 
\end{picture}
\end{center}

\medskip

Theorem 1 in Section 2 extends Theorem C for all dimensions. Note that if $L$ is a Sperner labeling, then $|\deg(L,\partial Q)|=1$. So Theorem C implies Theorem A and  from Theorem 1 follows Theorem B.

\section{Main results}

Here we consider the main theorem and its corollaries for a very general class of spaces $M$. One of very natural extensions of Theorem C is the case when $M$ is a polytope in ${\mathbb R}^d$.
 
In our papers \cite{Mus,MusSpT,MusKKM,MusArSp} we studied a more general case when $M$ is a piece-wise linear  manifold. 
In this case, if $M$ is a compact oriented  manifold with boundary, then from the one side it extends Theorem C for all dimensions and to a huge class of spaces (even in two dimensions), on the other side, almost all proofs can be easily transfer for this case. 

All results in this section hold for manifolds that admit quadrangulations. The class of such manifolds  is called  {piece-wise linear (PL)  manifolds.} Note that  a smooth manifold can be triangulated and quadrangulated, therefore it is also a PL manifold. However, there are topological manifolds  that do not admit a triangulation and therefore do not admit a quadrangulation. 
 
Every PL manifold $M$ admits a {\it quadrangulation:} that is, we can find a collection of cells  $Q$ of dimensions $0, 1,\ldots, d$, such that (1) any cell of dimension $k$ is homeomorphic to the $k$-cube $C^k$,  (2) any face of a cell belonging to $Q$ also belongs to $Q$, (3) any nonempty intersection of any two cells of $Q$ is a face of each, and (4) the union of the cells of $Q$ is $M$. (See details in \cite{Bryant}.)  Actually, a PL--manifold $M$ can be quadrangulated by many ways. 

Throughout this paper by word {\it manifold} we assume a compact PL manifold with or without boundary.  

Let $Q$ be a quadrangulation of $M$.  The vertex set of $Q$, denoted by $V(Q)$, is the union of the vertex sets of all cells of $Q$. 

\medskip 

Now for any labelling $c:V(C^n)\to \{0,1,\ldots,2^d-1\}$ we define a map $f_L:C^n\to C^d$. Since we have the bijection $c:V(C^d)\to \{0,1,\ldots,2^d-1\}$ this labelling can be considered as a map $L:V(C^n)\to V(C^d)$. 

\begin{lemma} For any labelling $L:V(C^n)\to V(C^d)$  we may uniquely define a multilinear  map $f_L:C^n\to C^d$ such that for each vertex $v\in V(C^n)$ we have $f_L(v)=L(v)$. 
\end{lemma}
\begin{proof} A multilinear map is a map of several variables that is linear separately in each variable. Actually, $f=(f_1,\ldots,f_d)$ is a multilinear map if each function $f_i(x_1,\ldots,x_n)$ is a multilinear function. A function $F(x_1,\ldots,x_n)$ is multilinear if 
$$
F(x_1,\ldots,x_n)=\sum\limits_{0\le w\le 2^n-1}{a_wx^w},$$ 
 where for $w=(p_1,\ldots,p_n),$ \, $p_k=0$ or 1,    
$
x^w:=x_1^{p_1}\cdot\cdot\cdot x_n^{p_n}. 
$

So a linear function  $F(x)=a_0+a_1x$ and a bilinear 
$$
F(x_1,x_2)=a_{00}+a_{01}x_1+a_{10}x_2+a_{11}x_1x_2.
$$

If a multilinear function $F(x_1,\ldots,x_n)$ is defined in $V(C^n)$, i. e. for $w\in V(C^n)$, $F(w)=b_w$ then the coefficients $a_w$ can be found. Indeed, for $n=1$ we have 
$$a_0=b_0, \; a_1=b_1-b_0.$$
For $n=2$: $$a_{00}=b_{00}, \, a_{01}=b_{01}-b_{00},\, a_{10}=b_{10}-b_{00},  \, a_{11}=b_{11}-b_{10}-b_{01}+b_{00}.$$
By induction these coefficients can be found for all $a_w$.  

 Let $u=(u_1,\ldots,u_n)$ and $w=(w_1,\ldots,w_n)$ be two binary numbers (0 and 1 strings). We write $u\preceq w$ if $u_i\le w_i$ for all $i=1,\ldots,n$. Denote by $h(u,w)$ the Hamming distance between $u$ and $w$, i. e. the number of digits in positions where they have different digit. Then 
$$
a_w=\sum\limits_{u\in V(C^n):u\preceq w}{(-1)^{h(u,w)}b_u}. 
$$
So all coefficients of the map $f_L:C^n\to C^d$ can be explicitly found. 
\end{proof}

\medskip 

\noindent{\bf Definition 1.} Let $Q$ be a quadrangulation of a $d$-dimensional manifold $M$. Let $L:V(Q)\to V(C^d)$ be a labelling. Every $n$-cell $\sigma$ in $Q$ is homeomorphic to $C^n$. So by Lemma 1 $f_L:\sigma\to C^d$ is uniquely defined.  It defines a piece-wise multilinear map $f_L:Q\to C^d$. 

\medskip

\noindent{\bf Definition 2.} Let $Q$ be a quadrangulation of a $d$-dimensional manifold $M$. Let $L:V(Q)\to V(C^d)$ be a labelling. 
We say that a cell $\sigma\in Q$ is {\it centrally labelled} if there is an internal point $x$ in $\sigma$ such that $f_L(x)=z$, where $z=(1/2,\ldots,1/2)$ is the center of $C^d$. 

\medskip	

\noindent{\bf Examples.} It is easy to see that if $d=2$, then there are two centrally labelled 1-cells with labels in the ends:  [(0,0),(1,1)] and [(0,1),(1,0)] and one 2-cell with labels  [(0,0),(1,0),(1,1),(0,1)]. 

For $d=3$ the set of centrally labelled cells is more complicated. We have three 1-cells with labels $[u,v]$, where $u+v=(1,1,1)$. There are two types of centrally labelled 2-cells: [(0,0,0),(1,0,0),(1,1,1),(0,1,1)] and [(0,0,0),(1,1,0),(1,0,1),(0,1,1)]. 

One of centrally labelled 3-cells is fully labelled, i. e. when $L:VC^3)\to V(C^3)$ and $L(w)=w$ for all $w\in V(C^3)$. There is a non-symmetric example. Let $L(w)=w$ for all $w\in V(C^3), \, w\ne(1,1,1)$, and $L(1,1,1)=(1,1,0)$. It is also a centrally labelled 3-cell. 

\medskip

An {\it oriented cube} is a cube $C^k$ together with a choice of one of its two possible orientations. 
An {\it oriented quadrangulation} $Q$ of an oriented $d$-dimensional manifold $M$ is $Q$ equipped with an orientation  on each $d$-cell and such that any two $d$-cells with a common $(d-1)$-face must have the same orientation. 
 
 The {\it degree of a continuous mapping} or {\it Brouwer's degree}  between two compact oriented manifolds of the same dimension is a number that represents the number of times that the domain manifold wraps around the range manifold under the mapping.   It is well known that the degree is a topological invariant (see, for instance, \cite{Milnor} and \cite[pp. 44--46]{Mat}).

Let us define the degree more rigorously. Let $f:M_1\to M_2$ be a continuous (piece-wise smooth) map, where $M_1$ and $M_2$ are compact $d$-dimensional manifolds without boundary. Let $y\in M_2$ be a regular value of $f$, that means $f^{-1}(y)=\{x_1,\ldots,x_n\}$ and  in a neighborhood of each $x_i$ the map $f$ is a local diffeomorphism. Then 
$$
\deg(f,y):=\sum\limits_{x\in f^{-1}(y)}{\sgn\,\det Df(x)},
$$
where $Df(x)$ is the Jacobi matrix of $f$ in $x$.
Actually, $\deg(f,y)$ does not depend on $y$ and so the value $\deg(f)$ is well defined. 

In a similar way, we could define the degree of a map between compact oriented manifolds with boundary. Let $M_i, \, i=1,2,$ be a compact manifolds with boundary $\partial M_i$. Suppose that $f:M_1\to M_2$ is such that $f(\partial M_1)\subseteq\partial M_2$. Denote $f^\partial:=f|_{\partial M_1}:\partial M_1\to\partial M_2$. It is well known that in this case 
$$
\deg(f)=\deg(f^\partial). 
$$ 
(A proof of this fact also can be fined in our paper \cite{MusArSp}, see the proof of Theorems 2.1 and Theorem 4.1.)

\medskip

\noindent{\bf Definition 3.} Let $Q$ be a quadrangulation of an oriented $d$-dimensional manifold $M$. Let $L:V(Q)\to V(C^d)$ be a labelling. We say that $L$ is a {\it neighboring labelling} and write $L\in\nl(\partial Q)$ if $f_L(\partial Q)\subseteq \partial C^d$. 

If $f_L(\partial Q)\subseteq \partial C^d$, then $f_L^\partial:\partial Q\to\partial C^d$ is well defined. Denote $\deg(f_L^\partial)$ by $\deg(L,\partial Q)$.  If $\partial Q$ is empty, i. e. $M$ is a manifold without boundary, then we set $\deg(L,\partial Q)=0$.

\begin{theorem} \label{SpQ} Let $Q$ be a quadrangulation of an oriented $d$-dimensional manifold $M$. Suppose $L:V(Q)\to V(C^d)$ be a labelling such that $L\in\nl(\partial Q)$. Then $Q$ contains at least $|\deg(L,\partial Q)|$ centrally labelled cells.
\end{theorem}
\begin{proof} Let $z$ denote the center of the cube $C^d$. Then $z=(1/2,\ldots,1/2)$ is an internal point of $C^d$. Since $$\deg(f_L,z)=\deg(f_L)=\deg(f_L^\partial)=\deg(L,\partial Q),$$ the set $f_L^{-1}(z)$ consists of  $n\ge |\deg(L,\partial Q)|$ points $\{x_1,\ldots,x_n\}$ in $M$. (Moreover, the assumption $L\in\nl(\partial Q)$ guarantees that these points cannot lie on the boundary $\partial M$.) 
	
Since $x_i$ cannot be a vertex of $Q$, there is a cell $c_i\in Q$ such that $x_i$ is an internal point of $c_i$. By definition, $c_i$ is centrally labelled. So we have $n$ centrally labelled cells.    
\end{proof}

Actually, Theorem C is a two-dimensional version of Theorem 1. It is easy to see that if $L$ is a Sperner labelling of a pile $P$, then $L\in\nl(\partial P)$ and $|\deg(L,\partial P)|=1$. Thus, Theorem 1 yields Theorem B as well as Theorem C implies Theorem A. 

\medskip

Not all manifolds can be oriented. For instance, the M\"obius strip is a non-orientable manifold with boundary. Theorem 1 can be extended for the non-orientable case. This extension is based on the concept of the degree of a continuous mapping modulo 2. Let $f:M_1\to M_2$ be a continuous map between two manifolds $M_1$ and $M_2$ of the same dimension. The degree  is a number that represents the number of times that the domain manifold wraps around the range manifold under the mapping. Then $\dg2(f)$ (the degree modulo 2) is 1 if this number is odd and 0 otherwise. It is well known that $\dg2(f)$ of a continuous map $f$  is a topological invariant modulo 2. 

\begin{theorem} \label{SpQ2}  Let $Q$ be a quadrangulation of a $d$-dimensional manifold $M$. Suppose $L:V(Q)\to V(C^d)$ be a labelling such that $L\in\nl(\partial Q)$ and $\dg2(L,\partial Q)\ne 0$. Then $Q$ contains at least one centrally labelled cell.
\end{theorem}

\medskip

\noindent{\bf Acknowledgment.} I  wish to thank Fr\'ed\'eric Meunier for references and helpful comments. 
  


 \medskip

\noindent O. R. Musin\\ 
Department of Mathematics, University of Texas at Brownsville, One West University Boulevard, Brownsville, TX, 78520 \\
and\\
IITP RAS, Bolshoy Karetny per. 19, Moscow, 127994, Russia\\ 
{\it E-mail address:} oleg.musin@utb.edu

\end{document}